\documentclass[12pt]{article}

\usepackage{amsthm, amsmath, amssymb,graphicx}

\numberwithin{equation}{section}

\theoremstyle{plain}	     
\newtheorem{thm}{Theorem}[section] 

\newtheorem{lem}[thm]{Lemma}

\theoremstyle{definition}

\theoremstyle{remark}

\begin{document}

\title{Two double-angle formulas of generalized trigonometric functions
\footnote{The work of S.\,Takeuchi was supported by JSPS KAKENHI Grant Number 17K05336.}}
\author{Shota Sato and Shingo Takeuchi \\
Department of Mathematical Sciences\\
Shibaura Institute of Technology
\thanks{307 Fukasaku, Minuma-ku,
Saitama-shi, Saitama 337-8570, Japan. \endgraf
{\it E-mail address\/}: shingo@shibaura-it.ac.jp \endgraf
}}

\date{}

\maketitle

\begin{abstract}
With respect to generalized trigonometric functions,
since the discovery of double-angle formula for a special case
by Edmunds, Gurka and Lang in 2012,
no double-angle formulas have been found. 
In this paper, we will establish new double-angle formulas 
of generalized trigonometric functions in two special cases.
\end{abstract}



\section{Introduction}

Let $p,\ q>1$ and 
$$F(x):=\int_0^x \frac{dt}{(1-t^q)^{1/p}}, \quad x \in [0,1].$$
We will denote by $\sin_{p,q}$ the inverse function of $F$, i.e.,
$$\sin_{p,q}{x}:=F^{-1}(x).$$
Clealy, $\sin_{p,q}{x}$ is an increasing function in $[0,\pi_{p,q}/2]$ to $[0,1]$,
where
$$\pi_{p,q}:=2F(1)=2\int_0^1 \frac{dt}{(1-t^q)^{1/p}}.$$
We extend $\sin_{p,q}{x}$ to $(\pi_{p,q}/2,\pi_{p,q}]$ by $\sin_{p,q}{(\pi_{p,q}-x)}$
and to the whole real line $\mathbb{R}$ as the odd $2\pi_{p,q}$-periodic 
continuation of the function. 
Since $\sin_{p,q}{x} \in C^1(\mathbb{R})$,
we also define $\cos_{p,q}{x}$ by $\cos_{p,q}{x}:=(d/dx)(\sin_{p,q}{x})$.
Then, it follows that 
$$|\cos_{p,q}{x}|^p+|\sin_{p,q}{x}|^q=1.$$
In case $(p,q)=(2,2)$, it is obvious that $\sin_{p,q}{x},\ \cos_{p,q}{x}$ 
and $\pi_{p,q}$ are reduced to the ordinary $\sin{x},\ \cos{x}$ and $\pi$,
respectively. 
This is a reason why these functions and the constant are called
\textit{generalized trigonometric functions} (with parameter $(p,q)$)
and the \textit{generalized $\pi$}, respectively. 

The generalized trigonometric functions are well studied in the context of 
nonlinear differential equations (see \cite{KT2019} and the references given there). 
Suppose that $u$ is a solution of 
the initial value problem of $p$-Laplacian
$$-(|u'|^{p-2}u')'=\frac{(p-1)q}{p} |u|^{q-2}u, \quad u(0)=0,\ u'(0)=1,$$
which is reduced to the equation $-u''=u$ of simple harmonic motion for
$u=\sin{x}$ in case $(p,q)=(2,2)$.  
Then, 
$$\frac{d}{dx}(|u'|^p+|u|^q)=\left(\frac{p}{p-1}(|u'|^{p-2}u')'+q|u|^{q-2}u\right)u'=0.$$
Therefore, $|u'|^p+|u|^q=1$, hence it is reasonable to define $u$ as 
a generalized sine function and $u'$ as a generalized cosine function.
Indeed, it is possible to show that $u$ coincides with $\sin_{p,q}$ defined as above.
The generalized trigonometric functions are often applied to
the eigenvalue problem of $p$-Laplacian.

Now, we are interested in finding double-angle formulas for 
generalized trigonometric functions.
It is possible to discuss addition formulas for these functions, 
but for simplicity we will not develop this point here.

No one doubts that the most basic formula is 
$$\sin_{2,2}{2x}=2\sin_{2,2}{x}\cos_{2,2}{x},\quad x \in \mathbb{R},$$
which is said to have been developed by Abu al-Wafa' (940--998),
a Persian mathematician and astronomer.
In case $(p,q)=(2,4)$, it is easy to see that $\sin_{2,4}{x}$ coincides with
the lemniscate function. Since this classical function has the double-angle
formula  (see, e.g. \cite[p.81]{PS1997}), 
which is due to Euler in 1751, we have
$$\sin_{2,4}{2x}=\frac{2\sin_{2,4}{x}\cos_{2,4}{x}}{1+\sin_{2,4}^4{x}},\quad
x \in \mathbb{R}.$$
The case $(p,q)=(3/2,3)$ goes back to the work of Dixon \cite{D1890} in 1890.
It is simple matter to check that $\sin_{3/2,3}{x}$ is identical to 
his elliptic function for $x \in [0,\pi_{3/2,3}/2]$, so that 
the double-angle formula of his function yields
\begin{equation}
\label{eq:dixon}
\sin_{3/2,3}{2x}=\frac{\sin_{3/2,3}{x}(1+\cos_{3/2,3}^{3/2}{x})}{\cos_{3/2,3}^{1/2}{x}(1+\sin_{3/2,3}^3{x})}, \quad x \in [0,\pi_{3/2,3}/4].
\end{equation}
Recently, Edmunds, Gurka and Lang \cite{EGL2012}
give a remarkable formula for $(p,q)=(4/3,4)$.
Function $\sin_{4/3,4}{x}$ can be written in terms of Jacobian
elliptic function, hence the double-angle formula of Jacobian elliptic function gives  
\begin{equation*}
\sin_{4/3,4}{2x}=\frac{2\sin_{4/3,4}{x}\cos_{4/3,4}^{1/3}{x}}{(1+4\sin_{4/3,4}^4{x}\cos_{4/3,4}^{4/3}{x})^{1/2}},\quad x \in [0,\pi_{4/3,4}/4].
\end{equation*}

As far as the generalized trigonometric functions are concerned, 
no other double-angle formulas have never been published. 

In this paper, we will deal with the cases $(p,q)=(2,3)$ and $(4/3,2)$.
The following double-angle formulas for the two special cases will be established.

\begin{thm}
\label{thm:(p,q)=(2,3)}
Let $(p,q)=(2,3)$. Then,
\begin{equation*}
\sin_{2,3}{2x}=\frac{4\sin_{2,3}{x}\cos_{2,3}{x}(3+\cos_{2,3}{x})^3}{(1+\cos_{2,3}{x})(8+\sin_{2,3}^3{x})^2},
\quad x \in [0,\pi_{2,3}/2].
\end{equation*}
\end{thm}

\begin{thm}
\label{thm:(p,q)=(4/3,2)}
Let $(p,q)=(4/3,2)$. Then, 
\begin{equation*}
\sin_{4/3,2}{2x}=\frac{4\sin_{4/3,2}{x}\cos_{4/3,2}^{1/3}{x}(1+\cos_{4/3,2}^{4/3}{x})}{(2\cos_{4/3,2}^{2/3}{x}+\sin_{4/3,2}^2{x})^2}, \quad x \in [0,\pi_{4/3,2}/2].
\end{equation*}
\end{thm}

The double-angle formulas for $\cos_{2,3}{x}$ and $\cos_{4/3,2}{x}$ are also
obtained by differentiating both sides of those for $\sin_{2,3}{x}$ and $\sin_{4/3,2}{x}$, respectively.
 
Finally, we summarize the relationship between parameters
in which double-angle formulas have been obtained
(Table \ref{tab:(p,q)}). 
Lemma \ref{lem:maf} (resp. Lemma \ref{lem:duality}) below
connects $(2,p)$ to $(p^*,p)$ (resp. $(p^*,2)$),
where $p^*:=p/(p-1)$.
Thus, there also exists an alternative proof of case $(4/3,4)$ such that 
one uses Lemma \ref{lem:maf} and the double-angle formula for $(2,4)$
(see \cite[Section 3.1]{T2016b}).
Nevertheless, the case $(3/2,2)$ is an open problem
because of difficulty of the inverse problem corresponding to \eqref{eq:f(2x)}.
\begin{table}[h]
\begin{center}
\begin{tabular}{|c|l|l|l|}
\noalign{\hrule height0.8pt}
$p$ & $(p^*,2)$ & $(2,p)$ & $(p^*,p)$ \\
\hline
$2$ & $(2,2)$ by Abu al-Wafa' & $(2,2)$ by Abu al-Wafa' & $(2,2)$ by Abu al-Wafa' \\
$3$ & $(3/2,2)$ open & $(2,3)$ \textbf{Theorem \ref{thm:(p,q)=(2,3)}} & $(3/2,3)$ by Dixon \\
$4$ & $(4/3,2)$ \textbf{Theorem \ref{thm:(p,q)=(4/3,2)}} & $(2,4)$ by Euler & $(4/3,4)$ by Edmunds et al. \\
\noalign{\hrule height0.8pt}
\end{tabular}
\end{center}
\caption{The parameters in which the double-angle formulas have been obtained.}
\label{tab:(p,q)}
\end{table}


\section{Proofs of theorems}

To prove Theorem \ref{thm:(p,q)=(2,3)}, we use the following multiple-angle formulas. 

\begin{lem}[\cite{T2016b}]
\label{lem:maf}
Let $p>1$ and $p^*:=p/(p-1)$. If $x \in [0,\pi_{2,p}/(2^{2/p})]=[0,\pi_{p^*,p}/2]$, then
\begin{align}
\sin_{2,p}{(2^{2/p}x)}&=2^{2/p}\sin_{p^*,p}{x}\cos_{p^*,p}^{p^*-1}{x}, \label{eq:mafs} \\
\cos_{2,p}{(2^{2/p} x)}&=\cos_{p^*,p}^{p^*}{x}-\sin_{p^*,p}^p{x} \notag \\
&=1-2\sin_{p^*,p}^p{x}=2\cos_{p^*,p}^{p^*}{x}-1. \label{eq:mafc}
\end{align}
\end{lem}

\begin{proof}[Proof of Theorem \ref{thm:(p,q)=(2,3)}]
From \eqref{eq:mafc} in Lemma \ref{lem:maf}, we have 
\begin{align}
\sin_{p^*,p}{x}=\left(\frac{1-\cos_{2,p}{(2^{2/p}x)}}{2}\right)^{1/p}, \label{eq:hankakus}\\
\cos_{p^*,p}{x}=\left(\frac{1+\cos_{2,p}{(2^{2/p}x)}}{2}\right)^{1/p^*}. \label{eq:hankakuc}
\end{align}

Let $x \in [0,\pi_{2,3}/2]$ and $y:=x/(2^{2/3})$. 
It follows from \eqref{eq:mafs} in Lemma \ref{lem:maf} that
since $2y \in [0,\pi_{2,3}/(2^{2/3})]=[0,\pi_{3/2,3}/2]$, 
\begin{align}
\sin_{2,3}{(2x)}
&=\sin_{2,3}{(2^{2/3}\cdot 2y)} \notag \\
&=2^{2/3}\sin_{3/2,3}{(2y)}\cos_{3/2,3}^{1/2}{(2y)}. 
\label{eq:xtoy}
\end{align}
Dixon's formula \eqref{eq:dixon}
with \eqref{eq:hankakus} and \eqref{eq:hankakuc} yields
\begin{align*}
\sin_{3/2,3}{(2y)}
&=\frac{\sin_{3/2,3}{y}(1+\cos_{3/2,3}^{3/2}{y})}{\cos_{3/2,3}^{1/2}{y}(1+\sin_{3/2,3}^3{y})}
=\frac{(1-C)^{1/3}(3+C)}{(1+C)^{1/3}(3-C)},
\end{align*}
where $C=\cos_{2,3}{(2^{2/3}y)}=\cos_{2,3}{x}$. Moreover,
$$\cos_{3/2,3}^{1/2}{(2y)}=(1-\sin_{3/2,3}^3{(2y)})^{1/3}=\frac{2^{4/3}C}{(1+C)^{1/3}(3-C)}.$$
Therefore, from \eqref{eq:xtoy} we have
\begin{align*}
\sin_{2,3}{(2x)}
&=2^{2/3}\cdot \frac{(1-C)^{1/3}(3+C)}{(1+C)^{1/3}(3-C)} \cdot \frac{2^{4/3}C}{(1+C)^{1/3}(3-C)}\\
&=\frac{4(1-C^2)^{1/3}C(3+C)^3}{(1+C)(9-C^2)^2}.
\end{align*}
Since $1-C^2=\sin_{2,3}^3{x}$, the proof is complete. 
\end{proof}

To show Theorem \ref{thm:(p,q)=(4/3,2)}, the following lemma is useful.

\begin{lem}[\cite{EGL2012}, \cite{T2016b}]
\label{lem:duality}
Let $p,\ q >1$. For $x \in [0,2]$,
\begin{gather*}
q\pi_{p,q}=p^*\pi_{q^*,p^*},\\
\sin_{p,q}{\left(\frac{\pi_{p,q}}{2}x\right)}
=\cos_{q^*,p^*}^{q^*-1}{\left(\frac{\pi_{q^*,p^*}}{2}(1-x)\right)}.
\end{gather*}
\end{lem}

\begin{proof}[Proof of Theorem \ref{thm:(p,q)=(4/3,2)}]
Let $x \in [0,\pi_{4/3,2}/2]$. Then,
since $4x/\pi_{4/3,2} \in [0,2]$, it follows from Lemma \ref{lem:duality} that
$$\sin_{4/3,2}{2x}
=\cos_{2,4}{\left(\frac{\pi_{2,4}}{2}\left(1-\frac{4x}{\pi_{4/3,2}}\right)\right)}
=\cos_{2,4}{\left(\frac{\pi_{2,4}}{2}-x\right)}.$$
Thus,
\begin{equation}
\label{eq:sin2x}
\sin_{4/3,2}{2x}=\sqrt{1-\sin_{2,4}^4{\left(\frac{\pi_{2,4}}{2}-x\right)}}.
\end{equation}
Function $\sin_{2,4}$ coincides with the lemniscate function, it has 
the addition formula: for any $u,\ v \in \mathbb{R}$,
\begin{equation}
\label{eq:atlem}
\sin_{2,4}{(u+v)}=\frac{\sin_{2,4}{u}\cos_{2,4}{v}+\sin_{2,4}{u}\cos_{2,4}{v}}
{1+\sin_{2,4}^2{u}\sin_{2,4}^2{v}}.
\end{equation}
Applying \eqref{eq:atlem} to the right-hand side of \eqref{eq:sin2x}, we obtain
$$\sin_{4/3,2}{2x}=\frac{2\sin_{2,4}{x}}{1+\sin_{2,4}^2{x}}.$$

We need only consider case $x \in (0,\pi_{4/3,2}/2)=(0,\pi_{2,4})$.
Let $f(x):=\sin_{4/3,2}{x}$ and $g(x):=\sin_{2,4}{x}$. Then, $g(x) \neq 0$ and
\begin{equation}
\label{eq:f(2x)}
f(2x)=\frac{2g(x)}{1+g(x)^2}=\frac{2}{1/g(x)+g(x)}.
\end{equation}
Therefore, it is easy to see that 
\begin{align}
\frac{1}{g(x)}+g(x)
&=\frac{2}{f(2x)}, \notag \\
\frac{1}{g(x)^2}+g(x)^2
&=\frac{4}{f(2x)^2}-2, \label{eq:+} \\
\frac{1}{g(x)^2}-g(x)^2
&=\frac{4}{f(2x)}\sqrt{\frac{1}{f(2x)^2}-1}. \label{eq:-}
\end{align}
Moreover, letting $u=v=x/2$ in \eqref{eq:atlem}, we see that $g(x)$ satisfies 
\begin{equation}
\label{eq:g(x)}
g(x)=\frac{2g(x/2)\sqrt{1-g(x/2)^4}}{1+g(x/2)^4}.
\end{equation}
Thus, substituting \eqref{eq:g(x)} into \eqref{eq:f(2x)}, we obtain
\begin{align*}
f(2x)
&=\frac{4(1/g(x/2)^2+g(x/2)^2)\sqrt{1/g(x/2)^2-g(x/2)^2}}
{(1/g(x/2)^2+g(x/2)^2)^2+4(1/g(x/2)^2-g(x/2)^2)}.
\end{align*}
Since \eqref{eq:+} and \eqref{eq:-} hold true for $x$ replaced with $x/2$,
we can express $f(2x)$ in terms of $f(x)$, i.e.,
$$f(2x)=\frac{4f(x)(1-f(x)^2)^{1/4}(2-f(x)^2)}
{(f(x)^2+2\sqrt{1-f(x)^2})^2}.$$
Since $1-f(x)^2=\cos_{4/3,2}^{4/3}{x}$, the proof is complete.
\end{proof}



\end{document}